\documentclass[11pt,reqno]{amsart}
\usepackage{amsmath,amssymb,latexsym, url,graphicx,newtxtext,newtxmath,tikz}

\usepackage[margin=1.15in]{geometry}
\usepackage{booktabs}
\usepackage{xcolor}
\usepackage{enumitem}
\usepackage[colorlinks=true,linkcolor=blue!50!black,citecolor=blue!50!black,urlcolor=blue!50!black]{hyperref}

\newtheorem{theorem}{Theorem}[section]
\newtheorem{lemma}[theorem]{Lemma}

\newtheorem{corollary}[theorem]{Corollary}
\newtheorem{question}[theorem]{Question}
\theoremstyle{definition}
\newtheorem{definition}[theorem]{Definition}
\theoremstyle{remark}
\newtheorem{remark}[theorem]{Remark}

\newcommand{\R}{\mathbb{R}}
\providecommand{\vv}[1]{\overrightarrow{#1}}
\newcommand{\Kl}{\operatorname{Kl}}
\newcommand{\dist}{\operatorname{dist}}
\newcommand{\sgn}{\operatorname{sgn}}
\newcommand{\supp}{\operatorname{supp}}
\newcommand{\T}{\mathsf T}
\newcommand{\PM}{\{-1,+1\}}
\newcommand{\cA}{\mathcal A}
\newcommand{\cC}{\mathcal C}

\begin{document}

\title[Mutually touching cylinders]{Mutually touching cylinders: Littlewood's seven, and at most ten in general}

\author[K. Kitamura]{Kenta Kitamura}
\address{Intellectual Moonshine}
\email{kitamura.kenta.1988@gmail.com}

\author[M. Liu]{Mingchang Liu}
\address{Georgia Institute of Technology, 225 North Ave NW, Atlanta, GA 30332, USA}
\email{michael.liu2@gatech.edu}

\author[I. Mati\'c]{Ivan Mati\'c}
\address{Department of Mathematics, Baruch College, City University of New York, One Bernard Baruch Way, New York, NY 10010, USA}
\email{ivan.matic@baruch.cuny.edu}

\author[R. Radoi\v{c}i\'c]{Rado\v{s} Radoi\v{c}i\'c}
\address{Department of Mathematics, Baruch College, City University of New York, One Bernard Baruch Way, New York, NY 10010, USA}
\email{rados.radoicic@baruch.cuny.edu}

\author[A. Spomer]{Andreas Spomer}
\address{Universit\"at zu K\"oln}
\email{andreas.spomer1996@gmail.com}


\begin{abstract}
We give a computer-assisted proof that at most seven congruent infinite circular cylinders in $\mathbb{R}^3$ can touch pairwise. This answers a question of Littlewood from 1968. For cylinders of arbitrary radii we prove the upper bound ten. Whether ten can be attained remains open.
\end{abstract}

\maketitle
 
\section{Introduction}\label{sec:intro}

Denote by $C(\ell,r)$ the infinite circular cylinder in $\R^3$ with axis $\ell$ and radius $r>0$. Two cylinders $C(\ell,r)$ and $C(\ell',r')$ \emph{touch} if their interiors are disjoint and their boundaries meet. A \emph{touching family} is a finite set of cylinders that touch pairwise.

\begin{theorem}\label{thm:seven}
The maximum number of congruent cylinders that touch pairwise is seven.
\end{theorem}

\begin{theorem}\label{thm:ten}
At most ten cylinders, of arbitrary radii, can touch pairwise.
\end{theorem}

The following question remains open.
 
\begin{question}\label{question:ten}
Is there a configuration of ten infinite circular cylinders that touch pairwise?
\end{question}

Littlewood \cite{Lit68} asked for the maximum number of congruent cylinders that touch pairwise (see also \cite[Section~2.4, Problem~14]{BMP}). Bezdek \cite{Bez05} proved that at most $24$ congruent cylinders can touch pairwise. Boz\'oki, Lee, and R\'onyai \cite{BLR15} found seven. Koizumi \cite{Koi25} proved the upper bound $18$. Dillon, Koizumi, and Luo \cite{DKL25} reduced it to $10$. H\"ofer \cite{Hof26} reduced it to $9$. 
For arbitrary radii, Pikhitsa and Pikhitsa \cite{PP19} argued, by methods they describe as heuristic, that at most ten cylinders with any smooth cross-section can touch pairwise.
Solymosi and Zahl \cite{SZ25} proved that touching families are finite, with the bound $192\,080$ in $\R^3$. Georgiev, G\'omez-Serrano, Tao, and Wagner \cite{GGTW25} found nine pairwise touching cylinders of different radii numerically.

Our proofs of Theorems~\ref{thm:seven} and~\ref{thm:ten} follow similar paths. We attach a sign to each pair of cylinders and a sign to each triple of cylinders. Then we show that certain sign configurations on four cylinders are impossible, whatever the radii are. Two further configurations, on four and on five cylinders, are impossible when all radii are equal.

Once these obstructions are established, the geometry drops out. What remains is a finite problem: an element of $\{-1,1\}$ is assigned to every pair and every triple of $\{0,\dots,n-1\}$ so that no forbidden configuration occurs. We prove that for $n=11$ every assignment contains a configuration that is forbidden for all radii, and that for $n=8$ every assignment contains one of the four forbidden configurations. This gives both upper bounds.

On eleven points there are $2^{55}$ assignments of signs to pairs and $2^{165}$ assignments of signs to triples. Checking the $2^{220}$ combinations one by one, at a billion per second, would take more than $10^{49}$ years. For a fixed assignment to pairs, the assignments to triples form the leaves of a binary tree. Each forbidden configuration involves at most ten of the $165$ signs, so it is detected long before a leaf is reached, and the subtree below it is discarded. Symmetries, and the passage from ten points to eleven, reduce the number of trees to be searched to a few thousand. The whole computation takes about a minute \cite{Verifier}.

The sign of a pair of cylinders goes back to Pikhitsa and Choi \cite{PC13}, who called it chirality. 
Pikhitsa and Pikhitsa \cite{PP17} observed that the matrix of these signs is a Seidel matrix. 
They proposed that certain small sign patterns cannot occur. In \cite{PP19} they used three such patterns to argue that eleven cylinders with any smooth cross-section cannot touch pairwise. For circular cylinders, Theorem~\ref{thm:ten} proves this.
Koizumi \cite{Koi25} identified the sign of a pair with the sign of the Klein product of the axes. The bounds of \cite{Koi25} and \cite{DKL25} are proved with these signs. H\"ofer's bound \cite{Hof26} is proved by a quadratic form argument. For our proof the signs of pairs alone are not enough. The sign of a triple of cylinders is the sign of the mixed product of the directions of the axes. Lemma~\ref{lem:master} ties the two kinds of signs together. Theorem~\ref{thm:seven} has been formalized in Lean~4 by the first author \cite{Kitamura26}, and its upper bound by the second author \cite{Liu26}. The proofs in this paper are self-contained.

\section{Lines and cylinders}\label{sec:lines}
 
For $\vv x,\vv y,\vv z\in\R^3$ we write $\vv x\cdot\vv y$ for the dot product and $\vv x\times\vv y$ for the cross product. We write $[\vv x\,\vv y\,\vv z]=\vv x\cdot(\vv y\times\vv z)$ for the determinant of the matrix with columns $\vv x,\vv y,\vv z$. The cross product vanishes exactly for parallel vectors and has length $\|\vv x\|\,\|\vv y\|\sin\theta$. The determinant vanishes exactly when $\vv x,\vv y,\vv z$ lie in a plane through the origin. We use $(\vv x\times\vv y)\times\vv z=\vv y\,(\vv x\cdot\vv z)-\vv x\,(\vv y\cdot\vv z)$.

An \emph{oriented line} $\ell=\{\vv p+t\vv u:t\in\R\}$ with $\|\vv u\|=1$ has the \emph{direction} $\vv u$ and the \emph{moment} $\vv w=\vv p\times\vv u$. The moment does not depend on the choice of $\vv p\in\ell$. The pair $(\vv u,\vv w)$ satisfies $\|\vv u\|=1$ and $\vv u\cdot\vv w=0$. Conversely, every such pair comes from exactly one oriented line, the line through $\vv u\times\vv w$ with direction $\vv u$. We write $\ell=(\vv u,\vv w)$. Reversing the orientation replaces $(\vv u,\vv w)$ by $(-\vv u,-\vv w)$. For oriented lines $\ell=(\vv u,\vv w)$ and $\ell'=(\vv u',\vv w')$ put
$$\Kl(\ell,\ell')=\vv u\cdot\vv w'+\vv u'\cdot\vv w .$$
It is symmetric, and $\Kl(\ell,\ell)=2\,\vv u\cdot\vv w=0$.
 
\begin{lemma}\label{lem:kl}
Let $\ell=(\vv u,\vv p\times\vv u)$ and $\ell'=(\vv u',\vv p'\times\vv u')$. Then $\Kl(\ell,\ell')=(\vv p-\vv p')\cdot(\vv u\times\vv u')$. If $\vv u,\vv u'$ are not parallel, then $\dist(\ell,\ell')=|\Kl(\ell,\ell')|/\|\vv u\times\vv u'\|$.
\end{lemma}
 
\begin{proof}
The first identity follows from $\vv u\cdot(\vv p'\times\vv u')=-\vv p'\cdot(\vv u\times\vv u')$ and $\vv u'\cdot(\vv p\times\vv u)=\vv p\cdot(\vv u\times\vv u')$. The distance formula is the standard one for two skew lines.
\end{proof}
 
Hence two cylinders $C(\ell,r)$ and $C(\ell',r')$ with non-parallel axes touch if and only if
\begin{equation}\label{eq:touch}
\Kl(\ell,\ell')=\pm(r+r')\,\|\vv u\times\vv u'\| .
\end{equation}

\begin{lemma}\label{lem:coplanar}
Let $C(\ell_1,r_1),C(\ell_2,r_2),C(\ell_3,r_3)$ be pairwise touching cylinders with pairwise non-parallel axes. Then $\vv u_1,\vv u_2,\vv u_3$ are linearly independent.
\end{lemma}
 
\begin{proof}
Suppose $\vv u_1,\vv u_2,\vv u_3\perp\vv m$ with $\|\vv m\|=1$. Then $\vv u_i\times\vv u_j=\pm\|\vv u_i\times\vv u_j\|\,\vv m$. By Lemma~\ref{lem:kl}, $\dist(\ell_i,\ell_j)=|(\vv p_i-\vv p_j)\cdot\vv m|=|h_i-h_j|$ with $h_i=\vv p_i\cdot\vv m$. So $|h_i-h_j|=r_i+r_j$ for the three pairs. Order the lines so that $h_1<h_2<h_3$. Then $r_1+r_3=h_3-h_1=(h_3-h_2)+(h_2-h_1)=r_2+r_3+r_1+r_2$. So $r_2=0$, a contradiction.
\end{proof}
 
We need to bound the size of a touching family that contains two cylinders with parallel axes. The bound rests on the following fact about circles in the plane.
 
\begin{lemma}\label{lem:circles}
Let $k_1,\dots,k_m$ be pairwise tangent circles in the plane with disjoint interiors, $m\ge2$. Then $m\le4$. The number of lines tangent to all of them, with all circles on one side, is two if $m=2$, at most one if $m=3$, and zero if $m=4$.
\end{lemma}
 
\begin{proof}
This is a result in elementary geometry. It goes back to Descartes \cite{Des43}, and a modern account is in Coxeter \cite[pp.~13--16]{Cox69}.
\end{proof}
 
Three pairwise tangent circles can have a common tangent line. The circles of radii $1,\tfrac14,1$ with centers $(0,1)$, $(1,\tfrac14)$, $(2,1)$ are pairwise tangent, and all three touch the $x$-axis.
 
\begin{lemma}\label{lem:parallel}
A touching family containing two cylinders with parallel axes has at most six members.
\end{lemma}
 
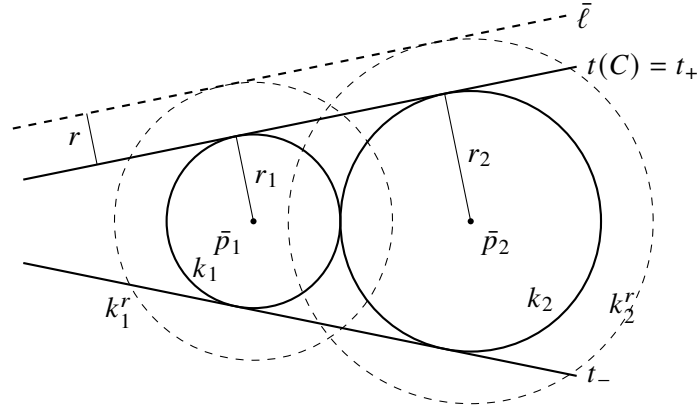
\begin{figure}[t]\centering
\begin{tikzpicture}[scale=1.15]
\coordinate (c1) at (0,0); \coordinate (c2) at (2.5,0);
\draw[dashed] (c1) circle (1.6);
\draw[dashed] (c2) circle (2.1);
\draw[thick] (c1) circle (1);
\draw[thick] (c2) circle (1.5);
\fill (c1) circle (1pt) node[below left] {$\bar p_1$};
\fill (c2) circle (1pt) node[below right] {$\bar p_2$};
\node at (-0.55,-0.55) {$k_1$};
\node at (3.3,-0.9) {$k_2$};
\node at (-1.55,-1.05) {$k_1^r$};
\node at (4.25,-1.0) {$k_2^r$};
\draw[thick] (-2.65,0.48) -- (3.72,1.78) node[right] {$t(C)=t_+$};
\draw[thick,dashed] (-2.77,1.068) -- (3.60,2.368) node[right] {$\bar\ell$};
\draw[thick] (-2.65,-0.48) -- (3.72,-1.78) node[right] {$t_-$};
\draw (-1.8,0.653) -- (-1.92,1.241) node[midway,left] {$r$};
\draw (c1) -- (-0.2,0.98) node[midway,right] {$r_1$};
\draw (c2) -- (2.2,1.47) node[midway,right] {$r_2$};
\end{tikzpicture}
\caption{The plane $\Pi$ in the proof of Lemma~\ref{lem:parallel}. The circles $k_1,k_2$ are the cross-sections of two parallel members $C_1,C_2$. The projected axis $\bar\ell$ of a member $C$ of radius $r$ is tangent to the enlarged circles $k_1^r,k_2^r$ (dashed). The line $t(C)$ is parallel to $\bar\ell$ at distance $r$, and is one of the two common tangents $t_\pm$ of $k_1,k_2$.}\label{fig:plane}
\end{figure}
 
\begin{proof}
Call a maximal set of mutually parallel members of the family a \emph{class}. Let $Q$ be a class with the most members, $q=|Q|$, and $\vv e$ its common direction. Then $q\ge2$. Let $\Pi$ be a plane perpendicular to $\vv e$, and write $\bar x$ for the projection of $\vv x$ to $\Pi$ (Figure~\ref{fig:plane}). A member $C_i\in Q$ meets $\Pi$ in a disc with center $\bar p_i$ and radius $r_i$. Let $k_i$ be its boundary circle. The circles $k_i$, $C_i\in Q$, are pairwise tangent with disjoint interiors. So $q\le4$ by Lemma~\ref{lem:circles}.
 
Let $C=C(\ell,r)$ be a member outside $Q$. The distance of a point $\vv x$ from $\ell_i$ is $\|\bar x-\bar p_i\|$. So $r+r_i=\dist(\ell,\ell_i)=\dist(\bar\ell,\bar p_i)$ for every $C_i\in Q$. That is, $\bar\ell$ is tangent to every circle $k_i^r$ with center $\bar p_i$ and radius $r_i+r$. Any two of these circles intersect, because the circles $k_i$ are tangent. So all of them lie on one side of $\bar\ell$. Let $t(C)$ be the line parallel to $\bar\ell$ at distance $r$ on that side. It is tangent to every $k_i$, $C_i\in Q$, with all of them on one side. By Lemma~\ref{lem:circles} there are at most two such lines if $q=2$, at most one if $q=3$, and none if $q=4$.
 
Let $C'=C(\ell',r')$ be another member outside $Q$ with $t(C')=t(C)=t$. The axes $\ell$ and $\ell'$ lie in the planes parallel to $t+\R\vv e$ at distances $r$ and $r'$ on the side away from the circles. Two non-parallel lines in parallel planes are at the distance of the planes. So $\ell\not\parallel\ell'$ would give $\dist(\ell,\ell')=|r-r'|<r+r'$. Hence the members outside $Q$ with the same line $t$ are mutually parallel, and there are at most $q$ of them. The family has at most $2+2+2$ members if $q=2$, at most $3+3$ if $q=3$, and $4$ if $q=4$.
\end{proof}

\begin{remark}
For congruent cylinders, H\"ofer \cite[Lemma~1(a)]{Hof26} proved that a touching family of at least five members contains no parallel pair. Lemma~\ref{lem:parallel} has a weaker conclusion, but it applies to arbitrary radii, and it is all that is used below.
\end{remark}
 
From now on, $C_1,\dots,C_n$ is a touching family with $n\ge7$, with axes $\ell_i=(\vv u_i,\vv w_i)$ and radii $r_i$. By Lemma~\ref{lem:parallel} the axes are pairwise non-parallel. By Lemma~\ref{lem:coplanar} any three directions are linearly independent. For $c\in\R^n$ let $\supp(c)=\{i:c_i\ne0\}$. In particular every nonzero $c\in K$ has $|\supp(c)|\ge4$, and $\Kl(\ell_i,\ell_j)\ne0$ for $i\ne j$ by \eqref{eq:touch}.

Let $L$ be the $3\times n$ matrix whose columns are the directions $\vv u_1,\dots,\vv u_n$, and $M$ the $3\times n$ matrix whose columns are the moments $\vv w_1,\dots,\vv w_n$. Let $K=\ker L\subseteq\R^n$, that is,
\begin{equation}\label{eq:K}
c\in K\quad\text{exactly when}\quad\sum_ic_i\vv u_i=\vv 0 .
\end{equation}
Let $N=L^{\T}M+M^{\T}L$ be the symmetric $n\times n$ matrix with entries $N_{ij}=\Kl(\ell_i,\ell_j)$. Its diagonal is zero.

\section{Signs and identities}\label{sec:signs}
 
Choose orientations of the axes. For distinct indices put
$$\sigma_{ij}=\sgn\Kl(\ell_i,\ell_j)\in\PM,\qquad \chi(a,b,c)=\sgn[\vv u_a\,\vv u_b\,\vv u_c]\in\PM .$$
The sign $\sigma_{ij}$ is symmetric. The sign $\chi$ changes under the exchange of two arguments. By \eqref{eq:touch},
\begin{equation}\label{eq:N}
N_{ij}=\Kl(\ell_i,\ell_j)=\sigma_{ij}\,(r_i+r_j)\,\|\vv u_i\times\vv u_j\| .
\end{equation}
Reversing the orientation of $\ell_m$ replaces $(\vv u_m,\vv w_m)$ by $(-\vv u_m,-\vv w_m)$. This changes the sign of $\sigma_{mj}$ for all $j\ne m$ and of $\chi(a,b,c)$ for all triples containing $m$, and nothing else.

\begin{lemma}\label{lem:cramer}
Let $\vv u_a,\vv u_b,\vv u_c,\vv u_d\in\R^3$ be vectors, no three in a plane through the origin. Put
$$k_a=[\vv u_b\,\vv u_c\,\vv u_d],\quad k_b=-[\vv u_a\,\vv u_c\,\vv u_d],\quad k_c=[\vv u_a\,\vv u_b\,\vv u_d],\quad k_d=-[\vv u_a\,\vv u_b\,\vv u_c] .$$
Then all four numbers are nonzero and $k_a\vv u_a+k_b\vv u_b+k_c\vv u_c+k_d\vv u_d=\vv 0$. Every linear dependence among the four vectors is a multiple of $(k_a,k_b,k_c,k_d)$.
\end{lemma}
 
\begin{proof}
Let $\vv x\in\R^3$. The $4\times4$ matrix with rows $(\vv x\cdot\vv u_i,\vv u_i^{\T})$, $i=a,b,c,d$, is singular, since its first column is a combination of the others. Expanding its determinant along the first column gives $\vv x\cdot(k_a\vv u_a+k_b\vv u_b+k_c\vv u_c+k_d\vv u_d)=0$ for all $\vv x$. The dependencies form a one-dimensional space because $\vv u_b,\vv u_c,\vv u_d$ is a basis.
\end{proof}
 
\begin{lemma}\label{lem:gp}
Let $\vv a,\vv v_1,\vv v_2,\vv v_3,\vv v_4\in\R^3$ and $\Delta_{ij}=[\vv a\,\vv v_i\,\vv v_j]$. Then $\Delta_{12}\Delta_{34}-\Delta_{13}\Delta_{24}+\Delta_{14}\Delta_{23}=0$.
\end{lemma}
 
\begin{proof}
If $\vv a=\vv 0$ all terms vanish. Otherwise complete $\vv a$ to a basis $\vv e_1,\vv e_2,\vv a$ and write $\vv v_i=f_i\vv e_1+g_i\vv e_2+h_i\vv a$. Then $\Delta_{ij}=m_{ij}[\vv a\,\vv e_1\,\vv e_2]$ with $m_{ij}=f_ig_j-f_jg_i$, and the identity reduces to $m_{12}m_{34}-m_{13}m_{24}+m_{14}m_{23}=0$. This is checked by expansion.
\end{proof}

For a four-element set $Q=\{a<b<c<d\}$ of indices, let $k^Q$ be the vector $(k_a,k_b,k_c,k_d)$ of Lemma~\ref{lem:cramer} for $\vv u_a,\vv u_b,\vv u_c,\vv u_d$, extended by zeros to a vector in $K$. Its entries are nonzero, and their signs are read off from $\chi$:
\begin{equation}\label{eq:kappa}
\sgn k^Q_a=\chi(b,c,d),\quad\sgn k^Q_b=-\chi(a,c,d),\quad\sgn k^Q_c=\chi(a,b,d),\quad\sgn k^Q_d=-\chi(a,b,c) .
\end{equation}
 
For $c,d\in\R^n$ and an index $j$ put
$$\rho_j(c,d)=\sum_{i\ne j}c_id_j\,\sigma_{ij}\,\|\vv u_i\times\vv u_j\| ,$$
the $j$-th \emph{row sum}. It is zero when $d_j=0$. The $j$-th column sum of the same array is $\rho_j(d,c)$.
 
\begin{lemma}\label{lem:master}
Let $c,d\in K$.
\begin{enumerate}[label=(\alph*)]
\item $\displaystyle\sum_{i\ne j}c_id_jN_{ij}=0$.
\item $\displaystyle\sum_{j}r_j\big(\rho_j(c,d)+\rho_j(d,c)\big)=0$. If all radii are equal, then $\sum_j\rho_j(c,d)=0$.
\item Let $d_j\ne0$ and $g\in\PM$. Put $\varepsilon_i=\sigma_{ij}\sgn(c_id_j)$ for $i\in\supp(c)\setminus\{j\}$. If all but at most one of the $\varepsilon_i$ equal $g$, then $\rho_j(c,d)$ has the sign $g$.
\end{enumerate}
\end{lemma}
 
\begin{proof}
(a) The terms with $i=j$ may be added, since $N_{ii}=0$. Then the sum is $c^{\T}Nd=(Lc)\cdot(Md)+(Mc)\cdot(Ld)=0$.
 
(b) Substitute \eqref{eq:N} into (a) and split $r_i+r_j$. The terms with $r_j$ give $\sum_jr_j\rho_j(c,d)$. The terms with $r_i$ give $\sum_ir_i\sum_{j\ne i}c_id_j\sigma_{ij}\|\vv u_i\times\vv u_j\|=\sum_ir_i\rho_i(d,c)$, by the symmetry of $\sigma_{ij}$ and $\|\vv u_i\times\vv u_j\|$. If all radii equal $r$, then (a) reads $2r\sum_j\rho_j(c,d)=0$.
 
(c) For $i\in\supp(c)\setminus\{j\}$ put $\vv y_i=c_i\,\vv u_j\times\vv u_i$. These vectors are nonzero. They sum to $\vv u_j\times\sum_ic_i\vv u_i-c_j\,\vv u_j\times\vv u_j$, which is $\vv 0$ by \eqref{eq:K}. No two are parallel, since otherwise $\vv u_j,\vv u_i,\vv u_{i'}$ would lie in a plane. There are at least three of them, because $|\supp(c)|\ge4$. Since $c_i\|\vv u_i\times\vv u_j\|=\sgn(c_i)\|\vv y_i\|$,
$$\rho_j(c,d)=|d_j|\sum_i\varepsilon_i\|\vv y_i\| .$$
If all $\varepsilon_i$ equal $g$, this has the sign $g$. If all but one do, say $\varepsilon_k=-g$, then $\rho_j(c,d)=g\,|d_j|\big(\sum_{i\ne k}\|\vv y_i\|-\|\vv y_k\|\big)$. This has the sign $g$, because $\vv y_k=-\sum_{i\ne k}\vv y_i$ is a sum of at least two nonzero, pairwise non-parallel vectors, so the triangle inequality is strict.
\end{proof}
 
\begin{corollary}\label{cor:principle}
Let $c,d\in K$ be nonzero and $g\in\PM$. Consider the following two conditions.
\begin{enumerate}[label=(\roman*)]
\item There is an index $j\in\supp(d)$ such that at least two indices $i\in\supp(c)\setminus\{j\}$ satisfy $\sigma_{ij}\sgn(c_id_j)=g$.
\item There is an index $i\in\supp(c)$ such that at least two indices $j\in\supp(d)\setminus\{i\}$ satisfy $\sigma_{ij}\sgn(c_id_j)=g$.
\end{enumerate}
\begin{enumerate}[label=(\alph*)]
\item If all radii are equal, then (i) holds.
\item For arbitrary radii, (i) or (ii) holds.
\end{enumerate}
\end{corollary}
 
Condition (ii) is condition (i) with $c$ and $d$ exchanged.
 
\begin{proof}
If (i) fails, then for every $j\in\supp(d)$ all but at most one of the signs $\sigma_{ij}\sgn(c_id_j)$, $i\in\supp(c)\setminus\{j\}$, equal $-g$. By Lemma~\ref{lem:master}(c) every $\rho_j(c,d)$ with $j\in\supp(d)$ has the sign $-g$, and the other $\rho_j(c,d)$ are zero. So $\sum_j\rho_j(c,d)\ne0$. For equal radii this contradicts Lemma~\ref{lem:master}(b), which proves (a). If (ii) fails as well, then every $\rho_i(d,c)$ with $i\in\supp(c)$ has the sign $-g$ too, by Lemma~\ref{lem:master}(c) with $c$ and $d$ exchanged. So $\sum_jr_j(\rho_j(c,d)+\rho_j(d,c))\ne0$, since all $r_j>0$. This contradicts Lemma~\ref{lem:master}(b) and proves (b).
\end{proof}
 
All constraints below apply Corollary~\ref{cor:principle} to the vectors $k^Q$. Part (b) gives constraints for arbitrary radii. Part (a) gives further constraints for equal radii.

\begin{lemma}\label{lem:four}
Let $Q$ be a set of four indices, $k=k^Q$, and $\varepsilon_{ij}=\sigma_{ij}\sgn(k_ik_j)$ for the six pairs $\{i,j\}\subseteq Q$. Let $P$ be the set of pairs with $\varepsilon_{ij}=+1$. Then $|P|\in\{2,3,4\}$. If $|P|=2$, the two pairs in $P$ share an element. If $|P|=4$, the two pairs not in $P$ share an element.
\end{lemma}
 
\begin{proof}
Apply Corollary~\ref{cor:principle}(b) with $c=d=k$. The signs $\sigma_{ij}\sgn(k_ik_j)$ are the $\varepsilon_{ij}$, and conditions (i) and (ii) coincide. So for each $g\in\PM$ some $j\in Q$ lies in at least two pairs with $\varepsilon_{ij}=g$. For $g=+1$ this excludes $|P|\le1$ and $|P|=2$ with two disjoint pairs. For $g=-1$ it excludes $|P|\ge5$ and $|P|=4$ with two disjoint pairs outside $P$.
\end{proof}

In the rest of this section all radii are equal. We scale so that they equal $\tfrac12$. Then $N_{ij}=\sigma_{ij}\|\vv u_i\times\vv u_j\|$ by \eqref{eq:N}.
 
\begin{lemma}\label{lem:two4}
Let $I$ and $J$ be sets of four indices. For each $g\in\PM$ there exists $j\in J$ such that at least two indices $i\in I\setminus\{j\}$ satisfy $\sigma_{ji}\sgn(k^J_jk^I_i)=g$.
\end{lemma}
 
\begin{proof}
This is Corollary~\ref{cor:principle}(a) with $c=k^I$ and $d=k^J$, since $\supp(k^I)=I$ and $\supp(k^J)=J$.
\end{proof}
 
\begin{lemma}\label{lem:five}
Let $s\in\PM$. There are no five members $\ell_0,\dots,\ell_4$ of the family such that
\begin{enumerate}[label=(\roman*)]
\item $\chi(0,1,2)=\chi(0,1,3)=\chi(0,1,4)=\chi(1,2,3)=\chi(1,2,4)=+1$ and $\chi(0,2,3)=\chi(0,3,4)=\chi(1,3,4)=-1$, and
\item $\sigma_{03}=-s$ and $\sigma_{ij}=s$ for the eight pairs
$$(i,j)\in\{(0,1),(0,2),(0,4),(1,2),(1,3),(1,4),(2,3),(2,4)\}.$$
\end{enumerate}
\end{lemma}
 
The signs $\sigma_{34}$, $\chi(0,2,4)$ and $\chi(2,3,4)$ are not restricted.
 
\begin{proof}
Suppose such lines exist. Write $[abc]$ for $[\vv u_a\,\vv u_b\,\vv u_c]$ and put
$$A=[123],\quad B=-[023],\quad C=[013],\quad D=[012],\quad P=[124],\quad Q=-[024],\quad R=[014].$$
By (i), $A,B,C,D,P,R$ are positive. Lemma~\ref{lem:gp} with $\vv a=\vv u_0$ gives $[012][034]-[013][024]+[014][023]=0$. Hence $C[024]=D[034]-RB$. Both terms on the right are negative by (i). So $Q>0$.
 
By Lemma~\ref{lem:cramer} applied to $\vv u_0,\vv u_1,\vv u_2,\vv u_3$ and to $\vv u_0,\vv u_1,\vv u_2,\vv u_4$,
\begin{equation}\label{eq:u3u4}
A\vv u_0+B\vv u_1+C\vv u_2=D\vv u_3,\qquad P\vv u_0+Q\vv u_1+R\vv u_2=D\vv u_4 .
\end{equation}
Substituting into $[134]$ and $[034]$ and expanding,
\begin{align*}
D^2[134]&=[\vv u_1,\,A\vv u_0+C\vv u_2,\,P\vv u_0+R\vv u_2]=(CP-AR)D,\\
D^2[034]&=[\vv u_0,\,B\vv u_1+C\vv u_2,\,Q\vv u_1+R\vv u_2]=(BR-CQ)D .
\end{align*}
So $CP-AR<0$ and $CQ-BR>0$ by (i). Put $x=CP/(AR)$ and $y=CQ/(BR)$. Then $0<x<1<y$.
 
Rescale the directions: $\vv v_0=A\vv u_0$, $\vv v_1=B\vv u_1$, $\vv v_2=C\vv u_2$, $\vv v_3=D\vv u_3$, and $\vv v_4=(DC/R)\,\vv u_4$. All factors are positive. Multiplying the second equation in \eqref{eq:u3u4} by $C/R$,
\begin{equation}\label{eq:v3v4}
\vv v_3=\vv v_0+\vv v_1+\vv v_2,\qquad \vv v_4=x\,\vv v_0+y\,\vv v_1+\vv v_2 .
\end{equation}
Write $\vv v_i=s_i\vv u_i$ with $s_i>0$ and put $K_{ij}=s\,s_is_jN_{ij}=s\,\sigma_{ij}\|\vv v_i\times\vv v_j\|$. By (ii), $K_{ij}=\|\vv v_i\times\vv v_j\|$ for the eight pairs in (ii) and $K_{03}=-\|\vv v_0\times\vv v_3\|$. The two relations \eqref{eq:v3v4} are two linear dependencies among $\vv u_0,\dots,\vv u_4$, that is, two vectors in $K$ by \eqref{eq:K}. Lemma~\ref{lem:master}(a) with $c=d$ equal to each of them, multiplied by $s$, gives
\begin{align}
K_{01}+K_{02}+K_{12}-K_{03}-K_{13}-K_{23}&=0,\label{eq:K3}\\
xyK_{01}+xK_{02}+yK_{12}-xK_{04}-yK_{14}-K_{24}&=0.\label{eq:K4}
\end{align}
Put $\vv a=\vv v_1\times\vv v_2$, $\vv b=\vv v_2\times\vv v_0$, and $\vv c=\vv v_0\times\vv v_1$. From \eqref{eq:v3v4},
\begin{align*}
\vv v_0\times\vv v_3&=\vv c-\vv b, & \vv v_1\times\vv v_3&=\vv a-\vv c, & \vv v_2\times\vv v_3&=\vv b-\vv a,\\
\vv v_0\times\vv v_4&=y\vv c-\vv b, & \vv v_1\times\vv v_4&=\vv a-x\vv c, & \vv v_2\times\vv v_4&=x\vv b-y\vv a .
\end{align*}
Substituting into \eqref{eq:K3} and \eqref{eq:K4},
\begin{align}
L_1&:=\|\vv a\|+\|\vv b\|+\|\vv c\|+\|\vv b-\vv c\|-\|\vv a-\vv c\|-\|\vv a-\vv b\|=0,\label{eq:f1}\\
L_2&:=y\|\vv a\|+x\|\vv b\|+xy\|\vv c\|-x\|y\vv c-\vv b\|-y\|\vv a-x\vv c\|-\|x\vv b-y\vv a\|=0.\label{eq:f2}
\end{align}
Put $T_1=\|\vv a-x\vv c\|+(1-x)\|\vv c\|-\|\vv a-\vv c\|$ and $T_2=\|y\vv a-x\vv b\|+(y-x)\|\vv b\|-y\|\vv a-\vv b\|$. Both are nonnegative by the triangle inequality, since $\vv a-\vv c=(\vv a-x\vv c)-(1-x)\vv c$ and $y(\vv a-\vv b)=(y\vv a-x\vv b)-(y-x)\vv b$ with $0<x<1<y$. A direct expansion shows
$$yL_1-L_2=y\|\vv b-\vv c\|+x\|y\vv c-\vv b\|+y\,T_1+T_2 .$$
The left side is zero. Every term on the right is nonnegative. The term $y\|\vv b-\vv c\|$ is positive, because $\vv b-\vv c=-(\vv v_0\times\vv v_3)\ne\vv 0$. This is a contradiction.
\end{proof}

\section{Sign patterns}\label{sec:finite}
 
Let $V$ be a finite set. A \emph{sign pattern} on $V$ is a pair $(\sigma,\chi)$. Here $\sigma$ assigns a sign $\sigma_{ij}=\sigma_{ji}\in\PM$ to every pair of distinct elements of $V$, and $\chi$ assigns a sign $\chi(a,b,c)\in\PM$ to every triple of distinct elements, changing sign under the exchange of two arguments. For a four-element subset $Q=\{a<b<c<d\}$ put
$$\kappa^Q_a=\chi(b,c,d),\quad\kappa^Q_b=-\chi(a,c,d),\quad\kappa^Q_c=\chi(a,b,d),\quad\kappa^Q_d=-\chi(a,b,c).$$
These are the signs of $k^Q$ in \eqref{eq:kappa} when $\chi$ comes from a touching family.
 
\begin{definition}\label{def:adm}
Consider the following properties of a sign pattern $(\sigma,\chi)$ on $V$.
\begin{enumerate}[label=(P\arabic*)]
\item\label{P1} For every $a\in V$ and all distinct $v_1,v_2,v_3,v_4\in V\setminus\{a\}$, the three signs
$$\chi(a,v_1,v_2)\chi(a,v_3,v_4),\qquad-\chi(a,v_1,v_3)\chi(a,v_2,v_4),\qquad\chi(a,v_1,v_4)\chi(a,v_2,v_3)$$
are not all equal.
\item\label{P2} For every four-element subset $Q\subseteq V$, let $P$ be the set of pairs $\{i,j\}\subseteq Q$ with $\sigma_{ij}\kappa^Q_i\kappa^Q_j=+1$. Then $|P|=3$, or $|P|=2$ and the two pairs in $P$ share an element, or $|P|=4$ and the two pairs not in $P$ share an element.
\item\label{P3} For every five-element subset $S\subseteq V$, all four-element subsets $I,J\subseteq S$ and every $g\in\PM$, there is $j\in J$ such that at least two indices $i\in I\setminus\{j\}$ satisfy $\sigma_{ji}\kappa^J_j\kappa^I_i=g$.
\item\label{P4} For every injective map $p:\{0,\dots,4\}\to V$, every $f\in\PM^5$, and every $s\in\PM$, the pattern $(\sigma',\chi')$ on $\{0,\dots,4\}$ with $\sigma'_{ab}=f_af_b\sigma_{p(a)p(b)}$ and $\chi'(a,b,c)=f_af_bf_c\chi(p(a),p(b),p(c))$ does not satisfy both (i) and (ii) of Lemma~\ref{lem:five}.
\end{enumerate}
For a set $\cC$ of these properties, a pattern is \emph{$\cC$-admissible} if it has every property in $\cC$. We write $\cA_\cC(n)$ for the set of $\cC$-admissible patterns on $\{0,\dots,n-1\}$.
\end{definition}
 
\begin{theorem}\label{thm:general}
Let $C_1,\dots,C_n$ be a touching family with $n\ge7$, with chosen orientations of the axes and sign pattern $(\sigma,\chi)$ as in Section~\ref{sec:signs}. Then $(\sigma,\chi)$ satisfies \ref{P1} and \ref{P2}. If all radii are equal, it also satisfies \ref{P3} and \ref{P4}.
\end{theorem}
 
\begin{proof}
Let $(\sigma,\chi)$ be the sign pattern of the family. For \ref{P1}, apply Lemma~\ref{lem:gp} with $\vv a=\vv u_a$ and $\vv v_i=\vv u_{v_i}$. The three products are the signs of the three terms of the identity. These terms are nonzero, because any three directions are independent. Three nonzero numbers with sum zero do not have the same sign. \ref{P2} is Lemma~\ref{lem:four}, since $\kappa^Q_i=\sgn k^Q_i$ by \eqref{eq:kappa}. Let the radii be equal. \ref{P3} is Lemma~\ref{lem:two4}. For \ref{P4}, let $p$, $f$, and $s$ be given. The five cylinders $C_{p(0)},\dots,C_{p(4)}$, with the orientation of $\ell_{p(a)}$ reversed when $f_a=-1$, have the sign pattern $(\sigma',\chi')$. By Lemma~\ref{lem:five} it does not satisfy both (i) and (ii).
\end{proof}

\begin{lemma}\label{lem:symmetry}
Let $\cC$ be any set of the properties in Definition~\ref{def:adm}. If $(\sigma,\chi)$ is $\cC$-admissible on $V$, then so are the following patterns.
\begin{enumerate}[label=(\alph*)]
\item \emph{Relabelling.} $(\sigma\circ\pi,\chi\circ\pi)$ for a bijection $\pi$ of $V$, where $(\sigma\circ\pi)_{ij}=\sigma_{\pi(i)\pi(j)}$ and $(\chi\circ\pi)(a,b,c)=\chi(\pi(a),\pi(b),\pi(c))$.
\item \emph{Reorientation.} $(\sigma^f,\chi^f)$ for $f\in\PM^V$, where $\sigma^f_{ij}=f_if_j\sigma_{ij}$ and $\chi^f(a,b,c)=f_af_bf_c\chi(a,b,c)$.
\item \emph{Reflection.} $(-\sigma,\chi)$.
\item \emph{Restriction.} The restriction of $(\sigma,\chi)$ to any subset of $V$.
\end{enumerate}
\end{lemma}
 
\begin{proof}
Every property is a statement about subsets of at most five elements. This gives (d).

Property \ref{P4} says that no pattern obtained from $(\sigma,\chi)$ by relabelling, reorientation, and restriction to five elements satisfies (i) and (ii) of Lemma~\ref{lem:five}, for either value of $s$. The patterns obtained in this way from $(\sigma\circ\pi,\chi\circ\pi)$ or from $(\sigma^f,\chi^f)$ are the same as those obtained from $(\sigma,\chi)$. Condition (ii) for $-\sigma$ and $s$ is condition (ii) for $\sigma$ and $-s$. So \ref{P4} is preserved by (a), (b), and (c).

Let $f\in\PM^V$ and $f_Q=\prod_{q\in Q}f_q$. By the definition of $\kappa^Q$, the pattern $(\sigma^f,\chi^f)$ has $\kappa^Q_i$ replaced by $f_Qf_i\kappa^Q_i$. So under (b) the three products in \ref{P1} are all multiplied by $f_{v_1}f_{v_2}f_{v_3}f_{v_4}$, the product $\sigma_{ij}\kappa^Q_i\kappa^Q_j$ in \ref{P2} is unchanged, and $\sigma_{ji}\kappa^J_j\kappa^I_i$ in \ref{P3} is multiplied by $f_If_J$, which does not depend on $i$ and $j$ and is absorbed by $g$.

Relabelling permutes the choices of $a,v_1,\dots,v_4$ in \ref{P1}, of $Q$ in \ref{P2}, and of $S,I,J$ in \ref{P3}. Since $\chi$ is alternating, listing $Q$ in another order changes $\kappa^Q$ by a global sign. This sign cancels in $\kappa^Q_i\kappa^Q_j$ and is absorbed by $g$ in \ref{P3}. So (a) preserves \ref{P1}, \ref{P2}, and \ref{P3}.

Reflection does not affect \ref{P1}. In \ref{P2} it exchanges $P$ with its complement, and the property is symmetric under this exchange. In \ref{P3} it changes every product by $-1$, absorbed by $g$.
\end{proof}
 
By Lemma~\ref{lem:symmetry}, whether some $\chi$ makes $(\sigma,\chi)$ $\cC$-admissible does not change when $\sigma$ is relabelled or reoriented.

\begin{theorem}\label{thm:finite}
$\cA_{\{P1,P2\}}(11)=\varnothing$ and $\cA_{\{P1,P2,P3,P4\}}(8)=\varnothing$.
\end{theorem}
 
Theorem~\ref{thm:finite} is proved by computer in Section~\ref{sec:verifier}.
 
\begin{proof}[Proof of Theorems~\ref{thm:seven} and~\ref{thm:ten}]
Suppose $C_1,\dots,C_{11}$ touch pairwise. By Theorem~\ref{thm:general} their sign pattern lies in $\cA_{\{P1,P2\}}(11)$, which is empty by Theorem~\ref{thm:finite}. Seven congruent cylinders that touch pairwise exist \cite{BLR15}. Suppose eight congruent cylinders touch pairwise. By Theorem~\ref{thm:general} their sign pattern lies in $\cA_{\{P1,P2,P3,P4\}}(8)$, which is empty.
\end{proof}
 
\section{The verifier}\label{sec:verifier}
 
Theorem~\ref{thm:finite} is proved by the program \cite{Verifier}. It is about five hundred lines of C++ and depends on nothing but the standard library. It proves two statements, which it calls Part~1 and Part~2. Part~1 is $\cA_{\{P1,P2\}}(11)=\varnothing$, a statement about $2^{220}$ sign patterns. Part~2 is $\cA_{\{P1,P2,P3,P4\}}(8)=\varnothing$, a statement about $2^{84}$ sign patterns. The program does not examine the patterns one by one. It rests on three ideas.

\emph{Pairs first, then triples.} For a fixed $\sigma$ on $m$ points, the unknowns are the $\binom m3$ signs $\chi(a,b,c)$ with $a<b<c$. The program assigns them one at a time by depth-first search. Each instance of a property in Definition~\ref{def:adm} involves at most ten unknowns: six for \ref{P1}, four for \ref{P2}, and the ten triples inside a five-element set for \ref{P3} and \ref{P4}. An instance is checked as soon as the last of its unknowns is assigned. If it fails, the search backtracks. The search ends either with a $\chi$ such that $(\sigma,\chi)$ is $\cC$-admissible, or with the conclusion that no such $\chi$ exists. In the first case we say that $\sigma$ \emph{survives}.

\emph{One point at a time.} If $\sigma$ on $m+1$ points survives, then so does its restriction to the first $m$ points, by Lemma~\ref{lem:symmetry}(d). So the program proceeds by the number of points. It starts with the one signing on two points. From the surviving signings on $m$ points it forms all signings on $m+1$ points whose restriction to the first $m$ points is one of them, and runs the search for each. Every surviving signing on $n$ points is reached in this way. Hence $\cA_\cC(n)=\varnothing$ if and only if no signing on $n$ points survives.

\emph{Symmetries.} If $\sigma$ survives, then so does every signing obtained from it by a relabelling and a reorientation, by Lemma~\ref{lem:symmetry}(a),(b). The program recognizes when two signings differ in this way. Among the signings it forms on $m$ points, it runs the search for one from each class and keeps only that one. The lists stay complete in the following sense: every surviving signing on $m$ points differs from a kept one by a relabelling and a reorientation. This is proved by induction on $m$. The restriction of a surviving $\sigma'$ to the first $m-1$ points survives, so it differs from a kept signing by a relabelling and a reorientation. The same relabelling, extended by fixing the new point, and a reorientation carry $\sigma'$ to a signing formed by the program on $m$ points. It survives, so it or a signing in its class was kept.
 
In Part~1, the numbers of signings kept on $3,4,\dots,11$ points are $2$, $3$, $5$, $10$, $20$, $30$, $25$, $4$, and $0$. In Part~2, the numbers kept on $3,4,\dots,8$ points are $2$, $3$, $5$, $10$, $12$, and $0$.
The run takes about a minute on a laptop. This proves Theorem~\ref{thm:finite}.

\section*{Acknowledgements}
The authors used Anthropic's Claude models (Claude Fable 5 and 5.1 and earlier Claude Opus models, accessed through claude.ai and Claude Code) and OpenAI's ChatGPT and Codex (GPT-6 Astra) between July and November 2026,
in developing the ideas of the proofs, in writing the verifier, and in editing the text. The authors take full responsibility for the correctness of the results.

\end{document}